\documentclass[psamsfonts,reqno]{amsart}

\usepackage{amssymb,amsfonts}
\usepackage{amsmath}
\usepackage[all,arc]{xy}
\usepackage{enumerate}
\usepackage{mathrsfs}
\usepackage{ytableau}
\usepackage{url}
\usepackage[margin=3cm]{geometry}
\usepackage{refcount}
\setrefcountdefault{1}

\newtheorem{thm}{Theorem}[section]
\newtheorem{cor}[thm]{Corollary}
\newtheorem{prop}[thm]{Proposition}
\newtheorem{lem}[thm]{Lemma}

\newtheorem{quest}[thm]{Question}

\theoremstyle{definition}
\newtheorem{defn}[thm]{Definition}

\newtheorem{exmp}[thm]{Example}

\newtheorem{notn}[thm]{Notation}

\theoremstyle{remark}
\newtheorem{rem}[thm]{Remark}

\newcommand{\cR}{\color{red}}
\newcommand{\cB}{\color{blue}}

\makeatletter

\newcommand\makebig[2]{
\@xp\newcommand\@xp*\csname#1\endcsname{\bBigg@{#2}}
\@xp\newcommand\@xp*\csname#1l\endcsname{\@xp\mathopen\csname#1\endcsname}
\@xp\newcommand\@xp*\csname#1r\endcsname{\@xp\mathclose\csname#1\endcsname}
}

\makeatother
\makebig{biggg}{3.0}
\makebig{Biggg}{3.5}
\makebig{bigggg}{4.0}
\makebig{Bigggg}{4.5}

\makeatletter

\let\c@equation\c@thm
\makeatother
\numberwithin{equation}{section}

\newcounter{saveexample}

\newcounter{savedthmcounter}
\newcounter{tobecontinuedthmcounter}

\makeatletter
\def\@extractthmnum#1.#2{#2}
\newcommand{\extractthmnum}[2]{
\setcounter{#1}{0}
\ifcsname r@#2\endcsname
\edef\@tempa{csname r@#2\endcsname}
\edef\@tempa{\expandafter\@firstoftwo\@tempa}
\setcounter{#1}{\expandafter\@extractthmnum\@tempa}
\fi
}

\makeatother

\title{The Benson - Symonds invariant for Ordinary and signed permutation modules}

\author{Aparna Upadhyay}
\address{Department of Mathematics\\ University at Buffalo, SUNY \\
244 Mathematics Building\\Buffalo, NY~14260, USA}

\email{aparnaup@buffalo.edu}

\date{14 August 2020}

\begin{document}

\begin{abstract}

The signed permutation modules are a simultaneous generalization of the ordinary permutation modules and the twisted permutation modules of the symmetric group. In a recent paper Dave Benson and Peter Symonds defined a new invariant $\gamma_G(M)$ for a finite dimensional module $M$ of a finite group $G$ which attempts to quantify how close a module is to being projective. In this paper, we determine this invariant for all the signed permutation modules of the symmetric group using tools from representation theory and combinatorics.\\
\smallskip

%\noindent \textbf{Keywords: Symmetric group, Signed Permutation module, tensor product.}
\smallskip

%\noindent 2010 \textit{Mathematics Subject Classification.} 20C30, 20C20, 05E10. 

\end{abstract}

\maketitle

%\tableofcontents

\section{\textbf{Introduction}}

Throughout this paper, let $\mathcal{S}_n$ denote the symmetric group. Suppose $\lambda \vdash a$ and $\mu \vdash b$ with corresponding Young subgroups $\mathcal{S}_\lambda$, $\mathcal{S}_\mu$ where $n=a+b$. Let $\mathbf{k}$ be a field of characteristic $p$. Define the signed permutation module by:
$$M^{(\lambda|\mu)} \cong \operatorname{Ind}_{\mathcal{S}_\lambda \times \mathcal{S}_\mu}^{\mathcal{S}_n} \mathbf{k} \boxtimes \operatorname{sgn}.$$

These modules are a simultaneous generalization of the ordinary permutation modules $M^\tau \cong M^{(\tau|\phi)}$ and the twisted permutation modules $M^\tau \otimes \operatorname{sgn} \cong M^{(\phi|\tau)}$. We determine the invariant for $M^{(\lambda|\mu)}$ as defined by Dave Benson and Peter Symonds in \cite{B_S}. This invariant is not yet known for ordinary permutation modules or twisted permutation modules. We determined the invariant for permutation modules of the symmetric group corresponding to two-part partitions in \cite{Aparna_BS}. In this paper we will prove \cite[Conjecture 6.1]{Aparna_BS} and answer \cite[Question 6.2]{Aparna_BS}. We will call this invariant the Benson - Symonds invariant and it will be denoted by $\gamma$. \\

\begin{defn}\cite{B_S} For a $kG$-module $M$, we write $M=M'\oplus (proj)$ where $M'$ has no projective direct summands and $(proj)$ denotes a projective module. Then $M'$ is called the $core$ of $M$ and denoted $core_G(M)$. We write $c_n^G(M)$ for the dimension of $core_G(M^{\otimes n})$.\\
\end{defn}

\begin{thm} \cite{B_S} If $M$ is a finite dimensional module of a finite group $G$ then $$\lim_{n \to \infty} \sqrt[n]{c_n^G(M)}$$ exists.
\end{thm}

Benson and Symonds define the invariant $\gamma_G(M)$ to be the above limit. The authors remark in \cite{B_S} that this invariant is difficult to compute, but captures interesting asymptotic properties of tensor products. It can also be used to measure the non-projective proportion of $M^{\otimes n}$ in the limit. This invariant is not known except for some examples the authors themselves computed using the computer algebra system \textbf{Magma} \cite{magma}. In this paper, we will see a closed formula for the invariant of an infinite class of modules.
\smallskip

\begin{rem} \label{count} The invariant $\gamma_G(M)$ is robust, in that $c_n^G(M)$ may be replaced by the number of indecomposable summands of $core_G(M^{\otimes n})$. Moreover, if the total number of non-isomorphic indecomposable summands that ever occur in a decomposition of $core_G(M^{\otimes n})$ is finite and a particular summand has the highest multiplicity for each $n$ then $c_n^G(M)$ can be replaced by the multiplicity of that particular summand.\\
\end{rem}

\begin{rem}\label{finitesummands} In the case of ordinary permutation modules, we know from \cite{KErdmannYoung} that the indecomposable summands of permutation modules are Young modules. More generally, we also know from the work of Donkin \cite{Donkinschur} that the indecomposable summands of the signed permutation modules are the signed Young modules which are labelled by pair of partitions $(\lambda|p\cdot\nu)$. As a consequence of Mackey's theorem, tensor product of (signed) Young modules decomposes into a direct sum of (signed) Young modules. Hence, only finitely many summands occur in tensor powers of (signed) permutation modules.\\% and consequently for $M_P$.\\
\end{rem}

\begin{notn} \label{notations}  Throughout this paper $n=kp+a_0$  for  $0 \leq a_0 < p$.
Let $\lambda=(\lambda_1,\lambda_2,...,\lambda_r)\vdash a $ and $\mu=(\mu_1,\mu_2,...,\mu_s) \vdash b$ with $a+b=n$ and $M=M^{(\lambda|\mu)}$. A $(\lambda|\mu)$-\textit{tableau} is a tableau of shape $(\lambda_1,\lambda_2,...,\lambda_r,\mu_1,\mu_2,...,\mu_s)$ with entries $1,2,...,n$ and a horizontal line separating the $\lambda$ and $\mu$ part. For example 

\begin{center}
\begin{ytableau}
\none[2] & \none [7] & \none[4]\\
\none[9] & \none[5]\\
\none[6] \\
\end{ytableau} \\
\rule{1.5cm}{0.4pt} \\
\begin{ytableau}
\none[1] & \none [8] & \none[3]  \\
\none[10]\\
\end{ytableau}
\end{center}

is a $(3,2,1|3,1)$-tableau. \\

\end{notn}

We aim to prove the following theorem which will also settle \cite[Conjecture 6.1]{Aparna_BS}:\\
\begin{thm} \label{myresult} Let $\lambda$ and $\mu$ be as in Notation \ref{notations} and $\mathbf{k}$ be a field of characteristic $p$, then
$$\gamma_{\mathcal{S}_n}(M^{(\lambda|\mu)})={n-p \choose \lambda_1-p,...\lambda_r,\mu_1,...,\mu_s}+{n-p \choose \lambda_1,\lambda_2-p,...,\lambda_r, \mu_1,...,\mu_s}+...+{n-p \choose \lambda_1,...,\lambda_r,\mu_1,...,\mu_s-p}.$$
\end{thm}
\smallskip
\smallskip

More generally, we will prove the following theorem:\\
\begin{thm} Let $M^{(\lambda|\mu)}$ be a signed permutation module of the symmetric group $\mathcal{S}_n$ over a field of characteristic $p$. Let $n=kp+a_0$ with $a_0<p$ and $P$ be a rank $k$ elementary abelian $p$-subgroup of $\mathcal{S}_n$ generated by $k$ different $p$-cycles. Let $M_P$ denote the restriction of $M^{(\lambda|\mu)}$ to $P$. We write $M_P=N \oplus Q$, where $N$ is the maximal possible summand of $M_P$ which is not $p$-faithful. Then, $\gamma_{\mathcal{S}_n}(M^{(\lambda|\mu)})$ is equal to the dimension of $N$.
\end{thm}

In the next section, we briefly summarize the fundamental results that will be used in the course of this paper. In section 3, we fix $P$ an elementary abelian $p$-subgroup of $\mathcal{S}_n$ of rank $k$ and consider the restriction of $M$ to $P$. We then investigate the indecomposable summands of this restriction and also determine the multiplicities of each indecomposable summand. We will also understand the properties of these indecomposable summands and their tensor products with each other. In section 4, we will compute the Benson - Symonds invariant for the restriction of $M$ to $P$. Finally in section 5, we will show that the Benson - Symonds invariant for $M$ is in fact equal to the Benson - Symonds invariant for $M$ restricted to $P$. In section 6, we put together some interesting observations and directions for further investigation.\\

\section{\textbf{Preliminaries}}

Given a $(\lambda|\mu)$-tableau $t$, let $R_t \leq \mathcal{S}_n$ be the row-stabilizer of $t$. A $(\lambda|\mu)$-tabloid $\{t\}$ is an equivalence class of tableaux under the action of $R_t$. %$M^{(\lambda|\mu)}$ is the module generated by the set of all $(\lambda|\mu)$-tabloids.
\\

\begin{defn} Let $t$ be a $(\lambda|\mu)$-tableau and $R_{t,\mu}$ denote a subgroup of $R_t$ which only permutes the rows of $\mu$. Let $\rho \in R_{t,\mu}$ be such that $\rho t$ has increasing rows in its $\mu$ part. Define:
$$\epsilon(t):=\operatorname{sgn} \rho.$$
\end{defn}

\smallskip

\begin{defn}\label{action} Let $t$ be a $(\lambda|\mu)$-tableau, $\pi \in \mathcal{S}_n$. Define a signed action $\circ$ of $\mathcal{S}_n$ by:
$$ \pi \circ \{t\} = \epsilon(t) \epsilon(\pi t) \{\pi t\}$$
where $\pi t$ is just the usual action of $\mathcal{S}_n$ acting on the entries of $t$.\\
\end{defn}

We will now see that this action is well-defined and affords a module isomorphic to $M^{(\lambda|\mu)}$.\\
\begin{lem}\label{welldefined}  Let $t$ be a $\lambda|\mu$-tableau, $\rho \in R_{t,\mu}$, $\pi \in \mathcal{S}_n$. Then: $$\operatorname{sgn} (\rho)\epsilon(\pi t)=\epsilon(\pi \rho t).$$
\end{lem}
\begin{proof}
Choose $\sigma \in R_{\pi t,\mu}$ such that $\sigma \pi t$ is $\mu$-row standard. Then $\epsilon(\pi t)=\operatorname{sgn}(\sigma)$ by definition. But $R_{\pi t,\mu}=\pi R_{t,\mu}\pi ^{-1}$ so:
\begin{equation} \label{2.4}
\sigma= \pi \alpha \pi ^{-1} \text{ for some } \alpha \in R_{t,\mu}.
\end{equation}
Since $\rho \in R_{t,\mu}$ therefore $R_{t,\mu}=R_{\rho t,\mu}$ and $\alpha \rho ^{-1} \in R_{\rho t,\mu}$. Thus $$\pi \alpha \rho ^{-1} \pi ^{-1} \in R_{\pi\rho t,\mu}.$$
But: 
\begin{eqnarray*}
\pi \alpha \rho ^{-1} \pi ^{-1} \pi \rho t &=& \pi \alpha t\\
&=& \sigma \pi t \text{ by } (\ref{2.4})
\end{eqnarray*}
is $\mu$-row standard. So by the definition of $\epsilon$:
\begin{eqnarray*}
\epsilon(\pi \rho t) &=& \operatorname{sgn} (\pi \alpha \rho^{-1}\pi^{-1})\\
&=&\operatorname{sgn}(\alpha) \operatorname{sgn}(\rho)\\
&=&\operatorname{sgn} (\sigma) \operatorname{sgn}(\rho) \text{ by } (\ref{2.4})\\
&=&\epsilon(\pi t) \operatorname{sgn}(\rho).
\end{eqnarray*}
\end{proof}
As a consequence, the action $\circ$ is well-defined:\\
\begin{lem} Suppose $\{s\}=\{t\}$ is a $(\lambda|\mu)$-tabloid and $\pi \in \mathcal{S}_n$. Then $\pi \circ \{s\}=\pi \circ \{t\}$.
\end{lem}
\begin{proof}
Suppose $\{t\}=\{\rho t\}$, where we can assume without loss of generality that $\rho \in R_{t,\mu}$, and let $\pi \in \mathcal{S}_n$. Then:
\begin{eqnarray*}
\pi \circ \{\rho t\} &=& \epsilon(\rho t)\epsilon(\pi\rho t)\{\pi \rho t\} \text{ by definition}.\\
&=&\epsilon(\rho t)\epsilon(\pi\rho t)\{\pi  t\} \text{ since }\rho \in R_t\\
&=&\epsilon(\rho t)\operatorname{sgn}(\rho) \epsilon(\pi t)\{\pi  t\} \text{ by Lemma }\ref{welldefined} \\
&=&\operatorname{sgn}(\rho) \epsilon(t)\operatorname{sgn}(\rho) \epsilon(\pi t)\{\pi  t\} \text{ setting } \pi =e \text{ in Lemma } \ref{welldefined} \\
&=&\epsilon(t) \epsilon(\pi t)\{\pi  t\}  \\
&=& \pi \circ \{t\}
\end{eqnarray*}
\end{proof}

We can now show:
\begin{prop} Definition \ref{action} makes the vector space spanned by all $(\lambda|\mu)$-tabloids into a $\mathbf{k}\mathcal{S}_n$ module isomorphic to $M^{(\lambda|\mu)}$.
\end{prop}
\begin{proof}
Calculating that $\pi_1 \circ (\pi_2\circ \{t\})=(\pi_1\pi_2)\circ \{t\}$ is straightforward from the definition of $\circ$. Now observe that our module has the correct dimension and is cyclic, generated by, for example, the tabloid with $1,2,...,n$ in order in successive rows. This generator spans a one-dimensional $\mathcal{S}_\lambda \times \mathcal{S}_\mu$-submodule isomorphic to $\mathbf{k} \boxtimes \operatorname{sgn}$, which implies \cite[p.56]{alperinlocal} the module is isomorphic to $\operatorname{Ind}_{\mathcal{S}_\lambda \times \mathcal{S}_\mu}^{\mathcal{S}_n} \mathbf{k} \boxtimes \operatorname{sgn}$.
\end{proof}

This signed action $\circ$ makes the set of all $(\lambda|\mu)$-tabloids a basis for $M^{(\lambda|\mu)}$ which therefore has the dimension equal to $${n \choose \lambda_1,...,\lambda_r,\mu_1,...,\mu_s}.  $$\\

\begin{notn}\label{signtabloid} As mentioned before in Notation \ref{notations}, a $(\lambda|\mu)$-tabloid is denoted by using a horizontal line to separate the $\lambda$ and $\mu$ part. For notational convenience, throughout this paper we will use two different colors to denote the entries of the $\lambda$ part and the $\mu$ part. For example \\
$$
\ytableausetup{boxsize=normal,tabloids}
\ytableaushort{{\cR 2}{\cR 7}{\cR 4}, {\cR 9}{\cR 5}, {\cR 6}, {\cB 1}{\cB 8}{\cB 3}, {\cB 10}} $$
is a $(3,2,1|3,1)$-tabloid. Hereafter, in this paper $t$ will be used to denote a tabloid (instead of $\{t\}$) and it will be understood that the representative of the equivalence class is chosen such that the entries of each row are increasing.\\
\end{notn}

Let us now recall some results by Benson and Symonds that will be used below.

\begin{thm}\label{maxElt} \cite[Theorem 7.2]{B_S} Let $M$ be a $kG$-module. Then $$\gamma_G(M)=\max_{E\leq G} \gamma_E(M)$$ where the maximum is taken over the set of elementary abelian $p$-subgroups $E$ of $G$.
\end{thm}
\smallskip

\begin{thm} \cite[Lemma 2.10]{B_S} If $H$ is a subgroup of $G$ and $M$ is a $kG$-module then $\gamma_H(M) \leq \gamma_G(M)$.
\end{thm}
\smallskip
Therefore, while using \ref{maxElt} it is enough to take the maximum over the maximal elementary abelian $p$-subgroups of $G$.\\

Recall the conjugacy classes of maximal elementary abelian $p$-subgroups in $\mathcal{S}_n$. Let $(\mathbb{Z}/p)^n \cong V_n(p) \hookrightarrow \mathcal{S}_{p^n}$ using notations from \cite[p.185]{EltAbBook}. So, $V_n(p)$ is isomorphic to a rank $n$ elementary abelian $p$-subgroup of $\mathcal{S}_{p^n}$. We have the following theorem:\\

\begin{thm}\label{conjEltAb} \cite[p.185, Theorem 1.3]{EltAbBook} If $n=a_0 + i_1p + i_2 p^2 + ... + i_r p^r$ with $0 \leq a_0 < p$, $i_j \geq 0$ for $1 \leq j \leq r$, then there is a maximal elementary abelian $p$-subgroup of $\mathcal{S}_n$ corresponding to this decomposition 
$$\underbrace{V_1(p) \times ... \times V_1(p)}_{i_1} \times ... \times \underbrace{V_r(p) \times ... \times V_r(p)}_{i_r} $$
$$\subset \underbrace{\mathcal{S}_p \times ... \times \mathcal{S}_p}_{i_1} \times ... \times \underbrace{\mathcal{S}_{p^r} \times ... \times \mathcal{S}_{p^r}}_{i_r} \subset \mathcal{S}_n$$
and as we run over distinct decompositions $(i_1,...,i_r)$ these give the distinct conjugacy classes of maximal elementary abelian $p$-subgroups of $\mathcal{S}_n$.
\end{thm}

\smallskip

\smallskip

\section{\textbf{Indecomposable summands of $M^{(\lambda|\mu)}$ restricted to an elementary abelian $p$-subgroup}}

In this section, we analyse the indecomposable summands of $M$ when restricted to a maximal elementary abelian $p$-subgroup of $\mathcal{S}_n$ of rank $k$. We first consider a rank $k$ elementary abelian $p$-subgroup of $\mathcal{S}_n$ generated by $p$-cycles.\\
\begin{notn} \label{notation}
Let $P=<(1,2,...,p),(p+1,...,2p),...,((k-1)p+1,...,kp)>$. Let $M_P=M \downarrow _P$. The action of $P$ on $(\lambda|\mu)$-tabloids partitions the set of all $(\lambda|\mu)$-tabloids into $P$-orbits. Given a tabloid $t$, let $<t>$ denote the cyclic $kP$-module generated by $t$. Let $B_i=\{(i-1)p+1,...,ip\}$, for $1\leq i \leq k$ and $B_{k+j}=\{kp+j\}$, for $1 \leq j \leq a_0$. Then, $\{B_1,B_2,...,B_{k+a_0}\}$ is a partition of the set $\{1,2,...,n\}$. We will call each $B_i$ a block. Let $\tau_i=((i-1)p+1,...,ip) \in P \leq \mathcal{S}_n$. Note that $$|B_i| = \left\{\begin{array}{lr}
        p, & \text{when } 1 \leq i \leq k\\
        1, & \text{when } i>k
        \end{array}\right.$$

\end{notn}

\begin{defn} We say that a $(\lambda|\mu)$-tabloid $t$ is constituted by blocks $\{B_i\}_{i \in I}$, $I \subseteq \{1,...,k\}$ if $I$ is the smallest set such that for every element of the set $\{\tau_j\}_{j \in J}$, $J=\{1,...,k\} \setminus I$, $\tau_j \circ t =\pm t$.\\
\end{defn}

\begin{notn} Given a $(\lambda|\mu)$-tabloid $t$, let $I^t \subseteq \{1,...,k\}$ denote the indexing set such that $t$ is constituted by $\{B_i\}_{i \in I^t}$.\\
\end{notn}

\begin{lem} If $d=|I^t|$ then $\operatorname{dim} <t>=p^d$.\\
\end{lem}

\begin{cor} If $I^t=\{1,...,k\}$ then $\operatorname{dim} <t>=p^k$ and $<t>$ is a projective $\mathbf{k}P$-module.\\
\end{cor}

\begin{rem} \label{construct} Observe that for $n$, $k$ and $B_i$'s as above, given an indexing set $I \subseteq \{1,...,k\}$, a $(\lambda|\mu)$-tabloid $t$ constituted by blocks $\{B_i\}_{i \in I}$ can be constructed by filling in entries such that all the entries of the chosen blocks do not lie in the same row of $t$.
\end{rem}

Let us now look at an example:
\begin{exmp} \label{myexample}  Let $(\lambda|\mu)=(3,1|3,2)$ and $M^{(\lambda|\mu)}$ be the signed permutation module of the symmetric group $\mathcal{S}_n$ over a field $\mathbf{k}$ of characteristic 3. Let $P=<(1,2,3),(4,5,6),(7,8,9)>$.
$$\dim M^{(\lambda|\mu)}=\frac{9!}{3!1!3!2!}=5040.$$
Following Notation \ref{notation}, $B_1=\{1,2,3\},B_2=\{4,5,6\},B_3=\{7,8,9\}$ and $\tau_1=(1,2,3),\tau_2=(4,5,6),\tau_3=(7,8,9)$. Consider \\
$$t_1=
\ytableausetup{boxsize=normal,tabloids}
\ytableaushort{{\cR 4}{\cR 5}{\cR 6}, {\cR 1}, {\cB 7}{\cB 8}{\cB 9}, {\cB 2}{\cB 3}} $$
which is a $(3,1|3,2)$-tabloid constituted by the block $B_1$. The cyclic $\mathbf{k}P$-module generated by $t_1$ is 3-dimensional with following vectors as a basis:
$$ \ytableausetup{boxsize=normal,tabloids}
\ytableaushort{{\cR 4}{\cR 5}{\cR 6}, {\cR 1}, {\cB 7}{\cB 8}{\cB 9}, {\cB 2}{\cB 3}} \text{ , } - \ytableaushort{{\cR 4}{\cR 5}{\cR 6}, {\cR 2}, {\cB 7}{\cB 8}{\cB 9}, {\cB 3}{\cB 1}}\text{ , } \ytableaushort{{\cR 4}{\cR 5}{\cR 6}, {\cR 3}, {\cB 7}{\cB 8}{\cB 9}, {\cB 1}{\cB 2}} $$
Consider another tabloid; for instance, $$s_1=
\ytableausetup{boxsize=normal,tabloids}
\ytableaushort{{\cR 7}{\cR 8}{\cR 9}, {\cR 1}, {\cB 2}{\cB 3}{\cB 4}, {\cB 5}{\cB 6}} $$
which is a $(3,1|3,2)$-tabloid constituted by the blocks $B_1$ and $B_2$. The cyclic $\mathbf{k}P$-module generated by $t_2$ is 9-dimensional with following vectors as a basis:
$$ \ytableausetup{boxsize=normal,tabloids}
\ytableaushort{{\cR 7}{\cR 8}{\cR 9}, {\cR 1}, {\cB 2}{\cB 3}{\cB 4}, {\cB 5}{\cB 6}} \text{ , } - \ytableaushort{{\cR 7}{\cR 8}{\cR 9}, {\cR 2}, {\cB 3}{\cB 1}{\cB 4}, {\cB 5}{\cB 6}}\text{ , } \ytableaushort{{\cR 7}{\cR 8}{\cR 9}, {\cR 3}, {\cB 1}{\cB 2}{\cB 4}, {\cB 5}{\cB 6}} \text{ , } $$ 
\smallskip
$$ - \ytableaushort{{\cR 7}{\cR 8}{\cR 9}, {\cR 1}, {\cB 2}{\cB 3}{\cB 5}, {\cB 6}{\cB 4}} \text{ , }  \ytableaushort{{\cR 7}{\cR 8}{\cR 9}, {\cR 2}, {\cB 3}{\cB 1}{\cB 5}, {\cB 6}{\cB 4}}\text{ , }- \ytableaushort{{\cR 7}{\cR 8}{\cR 9}, {\cR 3}, {\cB 1}{\cB 2}{\cB 5}, {\cB 6}{\cB 4}} \text{ , } $$ 
\smallskip
$$ \ytableaushort{{\cR 7}{\cR 8}{\cR 9}, {\cR 1}, {\cB 2}{\cB 3}{\cB 6}, {\cB 4}{\cB 5}} \text{ , } - \ytableaushort{{\cR 7}{\cR 8}{\cR 9}, {\cR 2}, {\cB 3}{\cB 1}{\cB 6}, {\cB 4}{\cB 5}}\text{ , } \ytableaushort{{\cR 7}{\cR 8}{\cR 9}, {\cR 3}, {\cB 1}{\cB 2}{\cB 6}, {\cB 4}{\cB 5}} $$

\end{exmp}

\begin{rem}
\begin{enumerate}[(i)]
\item The work above gives us a direct sum decomposition of $M_P$, so that $$M_P=\bigoplus_{\alpha \in \mathcal{I}} A_\alpha$$
with $\operatorname{dim} A_\alpha = p^{d_\alpha}$ for some non-negative integer $d_\alpha$ where $\mathcal{I}$ is an indexing set counting the number of summands.\\
\item Each summand is cyclic and has as its basis a single orbit. We will show in Theorem \ref{summandindec} that they are indecomposable.\\
\end{enumerate}
\end{rem}

\noindent \textbf{Example 3.7. (cont.)} Coming back to $M^{(\lambda|\mu)}$ when $(\lambda|\mu)=(3,1|3,2)$. We have,

$$t_2=
\ytableausetup{boxsize=normal,tabloids}
\ytableaushort{{\cR 7}{\cR 8}{\cR 9}, {\cR 1}, {\cB 4}{\cB 5}{\cB 6}, {\cB 2}{\cB 3}} $$
is another tabloid constituted by the block $B_1$. Observe that
$$t_3=
\ytableausetup{boxsize=normal,tabloids}
\ytableaushort{{\cR 1}{\cR 2}{\cR 3}, {\cR 4}, {\cB 7}{\cB 8}{\cB 9}, {\cB 5}{\cB 6}} \text{ and }  t_4=
\ytableausetup{boxsize=normal,tabloids}
\ytableaushort{{\cR 7}{\cR 8}{\cR 9}, {\cR 4}, {\cB 1}{\cB 2}{\cB 3}, {\cB 5}{\cB 6}} $$
are tabloids constituted by the block $B_2$, and 
$$t_5=
\ytableausetup{boxsize=normal,tabloids}
\ytableaushort{{\cR 1}{\cR 2}{\cR 3}, {\cR 7}, {\cB 4}{\cB 5}{\cB 6}, {\cB 8}{\cB 9}}  \text{ and } t_6=
\ytableausetup{boxsize=normal,tabloids}
\ytableaushort{{\cR 4}{\cR 5}{\cR 6}, {\cR 7}, {\cB 1}{\cB 2}{\cB 3}, {\cB 8}{\cB 9}} $$ 
are tabloids constituted by the block $B_3$.\\
The six different cyclic $\mathbf{k}P$-submodules generated by the tabloids $t_1,...,t_6$ are all 3-dimensional and any other tabloid constituted by exactly one block is a basis element for one of these cyclic submodules.\\

\noindent Recall that $s_1$ is a tabloid constituted by the blocks $B_1$ and $B_2$. Here is a list of some other such tabloids:
$$s_2=
\ytableausetup{boxsize=normal,tabloids}
\ytableaushort{{\cR 7}{\cR 8}{\cR 9}, {\cR 4}, {\cB 1}{\cB 5}{\cB 6}, {\cB 2}{\cB 3}} \text{ , } s_3=
\ytableausetup{boxsize=normal,tabloids}
\ytableaushort{{\cR 7}{\cR 8}{\cR 9}, {\cR 1}, {\cB 3}{\cB 5}{\cB 6}, {\cB 2}{\cB 4}} \text{ , } s_4=
\ytableausetup{boxsize=normal,tabloids}
\ytableaushort{{\cR 7}{\cR 8}{\cR 9}, {\cR 4}, {\cB 2}{\cB 3}{\cB 6}, {\cB 1}{\cB 5}}  \text{ , }$$
\smallskip
$$s_5=
\ytableausetup{boxsize=normal,tabloids}
\ytableaushort{{\cR 7}{\cR 8}{\cR 9}, {\cR 1}, {\cB 2}{\cB 5}{\cB 6}, {\cB 3}{\cB 4}} \text{ , } s_6=
\ytableausetup{boxsize=normal,tabloids}
\ytableaushort{{\cR 7}{\cR 8}{\cR 9}, {\cR 4}, {\cB 2}{\cB 3}{\cB 5}, {\cB 1}{\cB 6}} \text{ , } s_7=
\ytableausetup{boxsize=normal,tabloids}
\ytableaushort{{\cR 2}{\cR 3}{\cR 4}, {\cR 1}, {\cB 7}{\cB 8}{\cB 9}, {\cB 1}{\cB 5}}  \text{ , }  s_8=
\ytableausetup{boxsize=normal,tabloids}
\ytableaushort{{\cR 1}{\cR 5}{\cR 6}, {\cR 4}, {\cB 7}{\cB 8}{\cB 9}, {\cB 2}{\cB 3}} \text{ , }  $$
\smallskip
$$
s_9=
\ytableausetup{boxsize=normal,tabloids}
\ytableaushort{{\cR 3}{\cR 5}{\cR 6}, {\cR 1}, {\cB 7}{\cB 8}{\cB 9}, {\cB 2}{\cB 4}} \text{ , } s_{10}=
\ytableausetup{boxsize=normal,tabloids}
\ytableaushort{{\cR 2}{\cR 3}{\cR 6}, {\cR 4}, {\cB 7}{\cB 8}{\cB 9}, {\cB 1}{\cB 5}}  \text{ , } s_{11}=
\ytableausetup{boxsize=normal,tabloids}
\ytableaushort{{\cR 2}{\cR 5}{\cR 6}, {\cR 1}, {\cB 7}{\cB 8}{\cB 9}, {\cB 3}{\cB 4}} \text{ , } s_{12}=
\ytableausetup{boxsize=normal,tabloids}
\ytableaushort{{\cR 2}{\cR 3}{\cR 5}, {\cR 4}, {\cB 7}{\cB 8}{\cB 9}, {\cB 1}{\cB 6}}.$$
\smallskip

Again the twelve different cyclic $\mathbf{k}P$-submodules generated by the tabloids $s_1,...,s_{12}$ are all 9-dimensional and any other tabloid constituted by the blocks $B_1$ and $B_2$ is a basis element for one of these cyclic submodules. Similarly, we can choose generators $u_1,...,u_{12}$ that are tabloids constituted by the blocks $B_1$ and $B_3$ and $v_1,...,v_{12}$ that are tabloids constituted by the blocks $B_2$ and $B_3$ such that every other tabloid constituted by exactly two blocks is a basis element for a cyclic $\mathbf{k}P$-submodule generated by one of our chosen tabloids. \\
\smallskip

\textbf{Tensor products of indecomposable summands}\\

\begin{notn} Let $A_i ^{(d)}$ be a direct summand of $M_P$ such that dim $A_i ^{(d)}=p^d$. Let $A_i ^{(d)}=<t_i>$, for some tabloid $t_i$. Let $A_Q$ denote the restriction of a $\mathbf{k}P$-module $A$ to a subgroup $Q$ of $P$.
\end{notn}

\begin{lem} \label{indecomp} If $t$ is constituted from the blocks $B_1,...,B_d$, then $<t>_Q \cong \mathbf{k}Q$ where $Q=<(1,2,...,p),(p+1,...,2p),...,((d-1)p+1,...,dp)>$.
\end{lem}
\begin{proof}
Let $\phi:\mathbf{k}Q \to <t>_Q$ be defined by $g \mapsto g \circ t$. This is an isomorphism. 
\end{proof}
\smallskip

\begin{cor} Let $A_i ^{(d)}=<t_i>$ and $A_j ^{(d)}=<t_j>$. $A_i ^{(d)} \cong A_j ^{(d)}$ if and only if both $t_i$ and $t_j$ are constituted by the same set of blocks indexed by $I$.
\end{cor}
\begin{proof}
Consider $Q$ a subgroup of $P$ generated by the $p$-cycles $\{\tau_\alpha\}_{\alpha \in I}$. Then, both $<t_i>_Q$ and $<t_j>_Q$ are isomorphic to $\mathbf{k}Q$. Note that $P=Q \times R$ for some $R$, therefore the isomorphism between $<t_i>_Q$ and $<t_j>_Q$ can be extended to an isomorphism between $<t_i>$ and $<t_j>$ since the elements of $P$ that are not in $Q$ act trivially on $<t_i>$ and $<t_j>$. The converse is obvious.
\end{proof}
\smallskip

\noindent \textbf{Example 3.7. (cont.)} Looking back at Example \ref{myexample}, we see that $<t_1> \cong <t_2>$, $<t_3> \cong <t_4>$ and $<t_5> \cong <t_6>$. Also
$$<s_i> \cong <s_j> \text{, for all } i,j \in \{1,...,12\}.$$
$$<u_i> \cong <u_j> \text{, for all } i,j \in \{1,...,12\}.$$
$$<v_i> \cong <v_j> \text{, for all } i,j \in \{1,...,12\}.$$

\begin{thm} \label{summandindec}  The cyclic summands of $M_P$ found above are indecomposable.
\end{thm}
\begin{proof}
Let $A$ be a summand of $M_P$. We know that $A$ is cyclic so let $A=<t>$ for some tabloid $t$. Then by Lemma \ref{indecomp} we know that $A_Q$ is isomorphic to $\mathbf{k}Q$, where $Q$ is the subgroup of $P$ generated by the $p$-cycles that constitute $t$. So $A_Q$ is an indecomposable $\mathbf{k}Q$-module. Hence, $A$ must be indecomposable as a $\mathbf{k}P$-module.
\end{proof}
\smallskip

\begin{rem} If $d=k$ then $A_i ^{(d)}$ is projective. Therefore if $A_i ^{(d)}$ is a non-projective summand then $d<k$.
\end{rem}
\smallskip

\noindent \textbf{Example 3.7. (cont.)} Consider the tabloid
$$z_1=
\ytableausetup{boxsize=normal,tabloids}
\ytableaushort{{\cR 1}{\cR 2}{\cR 4}, {\cR 3}, {\cB 5}{\cB 6}{\cB 7}, {\cB 8}{\cB 9}} $$
which is constituted by the blocks $B_1$, $B_2$ and $B_3$. The cyclic $\mathbf{k}P$-module generated by $z_1$ is a 27-dimensional projective $\mathbf{k}P$-module.

\smallskip

\begin{lem} \label{dimen} $A_i ^{(d)} \otimes A_i ^{(d)} \cong p^d A_i ^{(d)}$.
\end{lem}  
\begin{proof}
Consider $Q$ a subgroup of $P$ of rank $d$ generated by the $p$-cycles that correspond to the $d$ blocks that constitute $t$. Then, ${A_i ^{(d)}}_Q \cong \mathbf{k}Q$. So, we have ${A_i ^{(d)}}_Q \otimes {A_i ^{(d)}}_Q \cong p^d {A_i ^{(d)}}_Q$. This isomorphism can easily be extended to $P$ since the elements of $P$ not in $Q$ act trivially on $A_i ^{(d)}$ and $P=Q \times R$ for some $R$.
\end{proof}

\smallskip

\begin{lem} If $A_i ^{(d)} \ncong A_j ^{(d')}$ then the indecomposable summands of $A_i ^{(d)} \otimes A_j ^{(d')}$ are of the form $A_k ^{(e)}$, for some $e>\operatorname{max}\{d,d'\} $.
\end{lem}
\begin{proof}
Let $A_i ^{(d)}=<t_i>$ and $A_j ^{(d')}=<t_j>$ be constituted from the set $\mathcal{B}_i, \mathcal{B}_j \subset \{B_1,...,B_k\}$ respectively with $|\mathcal{B}_i|=d$ and $|\mathcal{B}_j|=d'$ and $\mathcal{B}_i \neq \mathcal{B}_j$. Let $e=|\mathcal{B}_i \cup \mathcal{B}_j|> \operatorname{max}\{d,d'\} $. Let $t_k$ be a tabloid constituted from the $e$ blocks in the set $\mathcal{B}_i \cup \mathcal{B}_j$. Then, $<t_k> \cong <t_i \otimes t_j>$ with the isomorphism $\psi: <t_k> \to <t_i \otimes t_j>$ given by $g \circ t_k \mapsto g \circ t_i \otimes g \circ t_j$.
\end{proof}

\smallskip

\noindent \textbf{Example 3.7. (cont.)} For $t_i$'s and $s_i$'s as in Example \ref{myexample}, we see that $$<t_1>\otimes <t_1> \cong <t_1> \otimes <t_2> \cong 3<t_1>$$
and $$<t_1> \otimes <t_3> \cong <s_1>$$
and so on.
\smallskip

\begin{cor} \label{whenproj} Let $\{t_{\alpha}\}_{\alpha \in \mathcal{I}}$ be a collection of tabloids constituted from the blocks $\{\mathcal{B}_{\alpha}\}_{\alpha \in \mathcal{I}}$, where $\mathcal{B}_{\alpha} \subset \{B_1,...,B_k\}$ for every $\alpha \in \mathcal{I}$ for any indexing set $\mathcal{I}$ such that $$\bigcup_{\alpha \in \mathcal{I}} \mathcal{B}_{\alpha}=\{B_1,...,B_k\}.$$ Then, $$<\bigotimes_{\alpha \in \mathcal{I}} t_{\alpha}>$$ is a projective $\mathbf{k}P$-module.
\end{cor}
\begin{proof}
If $$\bigcup_{\alpha \in \mathcal{I}} \mathcal{B}_{\alpha}=\{B_1,...,B_k\}.$$ Then, $$\operatorname{dim} <\bigotimes_{\alpha \in \mathcal{I}} t_{\alpha}> = p^k.$$ Hence it is a projective $\mathbf{k}P$-module.
\end{proof}
\smallskip

\section{\textbf{The Benson - Symonds invariant for $M_P$}}

In this section we calculate the Benson - Symonds invariant for $M_P$. We know from Remark \ref{count} that it is enough to count the number of non-projective summands of the $core$ of higher tensor powers of $M_P$ or record the multiplicity of the summand that occurs with highest coefficient in each tensor power for sufficiently large $n$.\\%  If we can identify a certain summand that has the highest multiplicity in each tensor power of $M_P$ then  \\

\noindent Let
\begin{equation}
core_P(M)=\bigoplus_{\alpha \in \mathcal{I}} A_\alpha \label{4.1}
\end{equation}
where each $A_\alpha$ is a non-projective indecomposable summand with $\operatorname{dim} A_\alpha = p^{d_\alpha}$ for some non-negative integer $d_\alpha$. $\mathcal{I}$ is an indexing set that counts the number of non-projective indecomposable summands of $M_P$. Let $\mathcal{B}_\alpha$ be the subset of $\{B_1,...,B_k\}$ containing the blocks that constitute the tabloids in $A_\alpha$. We have
$$core_P(M^{\otimes n}) = \bigoplus_{\nu \models n} \Bigg( {n \choose \nu} core \Big( \bigotimes_\alpha (A_\alpha)^{\otimes \nu _\alpha} \Big )\Bigg) $$
where ${n \choose \nu}$ is the multinomial coefficient and if $\nu_\alpha \geq 1$ then by Lemma \ref{dimen} $$(A_\alpha)^{\otimes \nu _\alpha}= p^{d_\alpha(\nu_\alpha-1)}\cdot A_\alpha.$$
Therefore for any $\mathcal{J} \subset \mathcal{I}$, with $w=|\mathcal{J}|$ the coefficient of $\bigotimes_{\alpha\in \mathcal{J}} A_\alpha$ is 
\begin{eqnarray}
\sum_{\substack {\nu \models n \\ \text{into w parts}}} {n \choose \nu} p^{\sum_{\alpha \in J} d_\alpha(\nu _\alpha -1)}. \label{4.2}
\end{eqnarray}

The coefficient above depends only on $w$ and $|\mathcal{B}_\alpha|=d_\alpha$. We know from Corollary \ref{whenproj} that $\bigotimes_{\alpha\in \mathcal{J}} A_\alpha$ is non-projective if and only if $$\bigcup_{\alpha \in \mathcal{J}}\mathcal{B}_\alpha \neq \{B_1,...,B_k\}.$$
In fact the coefficient will be maximum if $\mathcal{J}$ is chosen such that $$\bigcup_{\alpha \in \mathcal{J}}\mathcal{B}_\alpha =\{B_1,...,B_{k-1}\}$$
making (\ref{4.2}) sum over the largest possible set of compositions of $n$  while being a coefficient of a non-projective summand.\\

\noindent Let $J_0$ be the maximal possible subset of $\mathcal{I}$ such that 
\begin{center}
\begin{equation}
\bigcup_{\alpha \in J_0}\mathcal{B}_\alpha =\{B_1,...,B_{k-1}\}. \label{4.3}
\end{equation}
\end{center}

\begin{thm} \label{maxinvt} From Equation (\ref{4.1}) and following the notations used above, we have $$\gamma_P(M)=\sum_{\alpha \in J_0} p^{d_\alpha}$$
\end{thm}
\begin{proof}

Let $w=|J_0|$. Rearrange the indexing set $\mathcal{I}$ so that the first $w$ indices in $\mathcal{I}$ are those in $J_0$. 
Let
$$N_P=\bigotimes_{\alpha=1}^w A_\alpha.$$
Then, the multiplicity (or coefficient) of $N_P$ is maximum from Equation (\ref{4.2}) and Equation (\ref{4.3}). Recall from Remark \ref{count}, that %Recall from Remark \ref{count} that the Benson - Symonds invariant of $M_P$ is the limit of the $n$th root of the number of non-projective indecomposable summands in higher tensor powers of $M_P$. 
%If $N$ is a submodule with maximum multiplicity in every tensor power of $M$, then $\gamma_P(M)$ is equal to the limit of the $n$th root of the multiplicity of $N$ in higher tensor powers of $M$.
the Benson - Symonds invariant for $M_P$ is equal to the limit of the $n$th root of the multiplicity of $N_P$ in higher tensor powers of $M_P$. Observe that the coefficient of $N_P$ in $(M_P)^{\otimes n}$ is the same as the coefficient of $N_P$ in $$\Bigg( \bigoplus_{\alpha=1}^w A_\alpha \Bigg)^{\otimes n}. $$
Therefore, the limit of the $n$th root of the multiplicity of $N_P$ in $(M_P)^{\otimes n}$ is the same as the limit of the $n$th root of the multiplicity of $N_P$ in $$\Bigg( \bigoplus_{\alpha =1}^w A_\alpha \Bigg)^{\otimes n} $$ which by Remark \ref{count} is the same as the limit of the $n$th root of the number of non-projective indecomposable summands in $$\Bigg( \bigoplus_{\alpha=1}^w A_\alpha \Bigg)^{\otimes n} $$
which is 
\begin{eqnarray*}
&=&\sum_{\substack {\nu \models n }} {n \choose \nu} p^{\sum_{\substack{ \alpha \in J_0\\ \nu_\alpha \neq 0} } d_\alpha(\nu _\alpha -1)}\\
&=&\sum_{\substack {\nu \models n }} {n \choose \nu} p^{\sum_{\alpha \in J_0} d_\alpha \nu _\alpha} \cdot p^{-\sum_{\alpha \in J_0} d_\alpha}\\
&=&\frac{1}{p^{\sum_{\alpha \in J_0} d_\alpha}}    \sum_{\substack {\nu \models n }} {n \choose \nu} p^{\sum_{\alpha \in J_0} d_\alpha \nu _\alpha} \\
&=&\frac{1}{p^{\sum_{\alpha \in J_0} d_\alpha}} \bigg (\sum_{\alpha \in J_0} p^{d_\alpha}\bigg )^n\\
\end{eqnarray*}
and 
\begin{eqnarray*}
\gamma_P(M)= \lim _{n \to \infty} \bigg (  \frac{1}{p^{\sum_{\alpha \in J_0} d_\alpha}}  \bigg)^{1/n} \sum_{\alpha \in J_0} p^{d_\alpha} =\sum_{\alpha \in J_0} p^{d_\alpha}. 
\end{eqnarray*}
\end{proof}
Recall that $p^{d_\alpha}$ is the dimension of $A_\alpha$ whose tabloids are constituted by $\mathcal{B}_\alpha \subset \{B_1,...,B_{k-1}\}$. Hence, the Benson - Symonds invariant is really the number of tabloids that are constituted by blocks other than the block $B_k$. \\

A $(\lambda|\mu)$-tabloid $t$ is constituted by blocks other than the block $B_k$ if and only if $\tau_k \circ t = \pm t$. This is possible if and only if all the entries of the block $B_k$ lie in any same row of $t$. So, we count the number of $(\lambda|\mu)$-tabloids such that the position of $p$ special entries is fixed in any of the $r+s$ rows of the $(\lambda|\mu)$-tabloid. It is easy to see that this number is equal to 
\smallskip

$${n-p \choose \lambda_1-p,...\lambda_r,\mu_1,...,\mu_s}+{n-p \choose \lambda_1,\lambda_2-p,...,\lambda_r, \mu_1,...,\mu_s}+...+{n-p \choose \lambda_1,...,\lambda_r,\mu_1,...,\mu_s-p}.$$
\smallskip

Hence, we have calculated the invariant for $M_P$:
\smallskip
$$\gamma_P(M^{(\lambda|\mu)})={n-p \choose \lambda_1-p,...\lambda_r,\mu_1,...,\mu_s}+{n-p \choose \lambda_1,\lambda_2-p,...,\lambda_r, \mu_1,...,\mu_s}+...+{n-p \choose \lambda_1,...,\lambda_r,\mu_1,...,\mu_s-p}.$$
\smallskip
\\
\noindent \textbf{Example 3.7. (cont.)} We now look back at our Example \ref{myexample}. For $M=M^{(\lambda|\mu)}$, where $(\lambda|\mu)=(3,1|3,2)$. $$core_P(M)=2<t_1>\oplus 2<t_3> \oplus 2<t_5>\oplus 12<s_1> \oplus 12<u_1> \oplus 12<v_1>. $$
Following (\ref{4.3}), we have $J_0=\{1,2\}\cup \{3,4\} \cup\{7,8,...,18\}$ and using Theorem \ref{maxinvt}
 \begin{eqnarray}
 \gamma_P(M)&=&2(3)+2(3)+12(9)=120 \notag \\
 &=&{6 \choose 0,1,3,2}+{6 \choose 3,-2,3,2}+ {6 \choose 3,1,0,2}+{6 \choose 3,1,3,-1} \notag \\
&=&60+0+60+0.    \label{4.4}
 \end{eqnarray}

\smallskip

\section{\textbf{The Benson - Symonds invariant for $M$}}

In this section we first look at the invariant for $M$ restricted to an arbitrary maximal elementary abelian $p$-subgroup of $\mathcal{S}_n$ and then use Theorem \ref{maxElt} to determine the invariant for $M$. Let $E$ be an arbitrary elementary abelian $p$-subgroup of $\mathcal{S}_n$.  It is clear that rank $E \leq k$. Using Theorem \ref{conjEltAb}, consider the partition of $\{1,2,...,n\}$ into orbits of $E$ and denote it by $C_1,...,C_m,C_{m+1},...,C_{m+a_0}$ where $|C_i|=p^{c_i}$ for some $c_i \in \mathbb{Z}_{>0}$ for $1 \leq i \leq m$ and $|C_i|=1$ for $i>m$.\\ 
$$core_E(M)=\bigoplus_{\alpha \in \mathcal{I}} A_\alpha$$
where each $A_\alpha$ is a cyclic non-projective indecomposable summand generated by a single tabloid with $\operatorname{dim} A_\alpha = p^{d_\alpha}$ for some non-negative integer $d_\alpha$ (same as the size of $E$-orbit of the tabloid). Let $\mathcal{C}_\alpha$ be the subset of $\{C_1,...,C_m\}$ containing the blocks that constitute the tabloids in $A_\alpha$.\\

Choose $1 \leq i_0 \leq m$ such that $c_{i_0}=\min \{c_i | 1 \leq i \leq m\}$. Let $J$ be the maximal possible subset of $\mathcal{I}$ such that $$\bigcup_{\alpha \in J}\mathcal{C}_\alpha =\{C_1,...,C_m\} \setminus \{C_{i_0}\}.$$
Then as in Theorem \ref{maxinvt}, we have 
$$\gamma_E(M)=\sum_{\alpha \in J} p^{d_\alpha}.$$
So $\gamma_E(M)$ is equal to the number of tabloids in $M$ that are constituted by the blocks other than the block $C_{i_0}$. This can at most be the number of tabloids in $M$ constituted by the blocks other than $B_k$ in the previous section because $|B_k| \leq |C_{i_0}|$.\\
Therefore, $$\gamma_E(M) \leq \gamma_P(M).$$
Hence, we have proved the following theorem:\\

\begin{thm} Let $M^{(\lambda|\mu)}$ be the signed-permutation module over the symmetric group $\mathcal{S}_n$ over a field of characteristic $p$ and let $P$ be an elementary abelian $p$-subgroup of $\mathcal{S}_n$ of highest possible rank. Then,
$$\gamma_{\mathcal{S}_n}(M^{(\lambda|\mu)})=\gamma_P(M^{(\lambda|\mu)}).$$
\end{thm}

\noindent This completes the proof of Theorem \ref{myresult}.\\
\\
\noindent \textbf{Example 3.7. (cont.)} We can now determine the invariant for $M$ as a $\mathbf{k}\mathcal{S}_n$-module. From Theorem \ref{conjEltAb} we know that $E=<(1,2,3)(4,5,6)(7,8,9),(1,4,7)(2,5,8)(3,6,9)>$ is the only other maximal elementary abelian 3-subgroup of $\mathcal{S}_n$. $C_1=\{1,...,9\}$ is a single $E$-orbit, hence for any $(\lambda|\mu)$-tabloid $t$, the cyclic $\mathbf{k}E$-module generated by $t$ is 9-dimensional which is a projective $\mathbf{k}E$-module since $|E|=9$. Therefore $$core_E(M)=\varnothing.$$
This implies that $\gamma_E(M)=0$ and hence $\gamma_{\mathcal{S}_n}(M)=	120$ from (\ref{4.4}).\\

\smallskip

\section{\textbf{Further Directions}}

In this section we include some observations and further questions. The following theorem is a consequence of the work above:\\
\begin{thm} \label{thmdim} Let $M^{(\lambda|\mu)}$ be the signed permutation module of the symmetric group $\mathcal{S}_n$ over a field of characteristic $p$. Let $n=kp+a_0$ with $a_0<p$ and $P$ be a rank $k$ elementary abelian $p$-subgroup of $\mathcal{S}_n$ generated by $k$ different $p$-cycles. Let $M_P$ denote the restriction of $M^{(\lambda|\mu)}$ to $P$. We write $M_P=N \oplus Q$, where $N$ is maximal possible summand of $M_P$ which is not $p$-faithful. Then, $\gamma_{\mathcal{S}_n}(M^{(\lambda|\mu)})$ is equal to the dimension of $N$.
\end{thm}

\begin{proof}
$N$ is not $p$-faithful, therefore $\gamma_{P}(N)=\operatorname{dim} N \leq \gamma_{P}(M)$. Let $\tau \in P$ be an element of order $p$ that acts trivially on $N$. Since, $N$ is maximal summand which is not $p$-faithful, we know that $\tau$ acts non-trivially on every summand of $Q$. For every $Q_i$, a summand of $Q$ there exists $\{N_j\}$, a collection of summands of $N$, such that 
$$\otimes_{j}N_j \otimes Q_i$$ is $p$-faithful and hence a projective $\mathbf{k}P$-module. This term will appear in a sufficiently large tensor power of $M_P$. So in the study of the asymptotics of the direct sum decomposition of tensor powers of $M_P$, the contribution of each summand of $Q$ eventually vanishes. Hence $\gamma_{P}(M)=\operatorname{dim} N $. \\

To determine $\gamma_{\mathcal{S}_n}(M^{(\lambda|\mu)})$ we take the maximum over all elementary abelian $p$-subgroups. We saw in Section 5 that the maximum is attained at $P$. Hence, $\gamma_{\mathcal{S}_n}(M^{(\lambda|\mu)})=\dim N$.

\end{proof}

\begin{cor} \label{exttrivial} If $Y^{(\lambda|p \cdot \nu)}$ is the signed Young module corresponding to partitions $\lambda$ and $\nu$ then $$\gamma_{\mathcal{S}_n}(Y^{(\lambda|p\cdot\nu)})= \max_E \dim N_E $$ where $E$ is a maximal elementary abelian $p$-subgroup of $\mathcal{S}_n$ and $Y^{(\lambda|p\cdot\nu)}_E=N_E \oplus Q_E$ with $Y^{(\lambda|p\cdot\nu)}_E$ being the restriction of $Y^{(\lambda|p\cdot\nu)}$ to $E$ and $N_E$ being the maximal possible summand of  $Y^{(\lambda|p\cdot\nu)}_E$ which is not $p$-faithful.
\end{cor}

%\begin{rem} Corollary \ref{exttrivial} holds more generally for every trivial source module. 
%\end{rem}

Since Young modules are indecomposable summands of ordinary Permutation modules, we can try to find a recursive formula for the invariant of Young modules using the knowledge of the invariant for Permutation modules and the inequalities relating the invariant of Young modules and Permutation module.

The next class of modules which is extensively studied in the representation theory of symmetric groups is of Specht modules $S^\lambda$. The main obstruction to our method in the case of Specht modules is that the indecomposable summands of tensor products of Specht modules is not known. If we have the knowledge of the number of composition factors of higher tensor powers of $S^\lambda$ or even the number of composition factors of the socle of higher tensor powers of $S^\lambda$, there is some hope to work out the invariant.\\ 

\noindent The next theorem is a compilation of some quick observations:\\  

\begin{thm} Suppose $\lambda \vdash n$. Then we have the following results:
\begin{enumerate}
\item If $\lambda$ is $p$-restricted then $\gamma_{\mathcal{S}_n} (Y^\lambda)=0$.
\item If $\lambda=(n-1,1)$ then $\gamma_{\mathcal{S}_n}(Y^\lambda)=\operatorname{dim} Y^\lambda -p$.
\item Let $p=2$ and $\lambda=(p,p)$. Then $\gamma_{\mathcal{S}_n} (Y^\lambda)=2$.
\item Let $p>2$ and $\lambda=(p+a,p)$ for $0 \leq a <p$. Then $\gamma_{\mathcal{S}_n} (Y^\lambda)=1$.
\item Let $\lambda$ be a hook-shaped partition of $n$ with $p$-weight 0 or 1. Then, $\gamma_{\mathcal{S}_n} (S^\lambda)$ is equal to the $p$-weight of $\lambda$.\\
\end{enumerate}
\end{thm}
\begin{proof}
The authors in \cite{B_S} state that the invariant $\gamma_G(M)$ is similar in nature to the \textit{complexity} $c_G(M)$ of a module, which describes the polynomial rate of growth of a minimal resolution of $M$, or equivalently the rate of growth of the non-projective part of $M \otimes \Omega^n(M)$. In particular, $\gamma_G(M)=0$ if and only if $M$ is a projective $\mathbf{k}G$-module and $\gamma_G(M)=1$ if and only if $M$ is endotrivial. We know that the same is true for $c_G(M)$. Since $Y^\lambda$ is projective if and only if $\lambda$ is $p$-restricted from \cite{KErdmannYoung}, \textit{(1)} follows. \\
We know that corresponding to the partition $\lambda=(n-1,1)$,  
$$M^\lambda=\left\{\begin{array}{lr}
        Y^\lambda, &  p\mid n\\
        \mathbf{k} \oplus Y^\lambda , & p \nmid n\\
        \end{array}\right.$$ 
and $\gamma_{\mathcal{S}_n}(M^\lambda)=n-p$ from Theorem \ref{myresult}. We also have from \cite{B_S} that $\gamma_G(\mathbf{k} \oplus N)=1+\gamma_G(N)$. Hence, \textit{(2)} follows.\\
\textit{(3)} follows from the fact that $M^{(2,2)} \cong Y^{(2,2)}$.\\
Let $p>2$ and $\lambda=(p+a,p)$ for $0 \leq a <p$. Then we know from \cite[Proposition 8.2]{carl_m_nak_endo} that $Y^\lambda$ is endotrivial. Hence \textit{(4)} follows.\\
If $\lambda$ is a hook-shaped partition, then we know from \cite[Theorem 1.3]{lim2009complexity} that the complexity of the Specht module $S^\lambda$ is exactly the $p$-weight of $\lambda$. Since $c_G(M)=\gamma_G(M)$ when $M$ is projective or endotrivial, \textit{(5)} follows.
     
\end{proof} 

Theorem \ref{thmdim} and Corollary \ref{exttrivial} suggest that the next class of modules to consider could be that of \textit{trivial source modules}. 
\begin{thm} If $M$ is a trivial source module of a finite group $G$, then $\gamma_G(M)$ is an integer.
\end{thm}
\begin{proof}
Let $M$ be a trivial source module of a finite group $G$. Let $P$ be a Sylow $p$-subgroup of $G$. Therefore, the restriction of $M$ to $P$ is a permutation module. Let $\mathcal{B}$ be a basis for the restriction. Let $E$ be an elementary abelian $p$-subgroup of $P$. The action of $E$ on $\mathcal{B}$ partitions the set of basis vectors into $E$-orbits. If the action of $E$ on $\mathcal{B}$ is transitive then the restriction of $M$ to $E$ in projective and therefore $\gamma_E(M)=0$. Otherwise following similar arguments as in the Proof of Theorem \ref{thmdim}, we get that $\gamma_E(M)$ is equal to the dimension of the maximal possible summand on the restriction of $M$ to $E$ that is not $p$-faithful. Hence, $\gamma_E(M)$ is an integer. The theorem follows by taking the maximum over all elementary abelian $p$-subgroups and using Theorem \ref{maxElt}.
\end{proof}

We have therefore answered \cite[Question 6.2]{Aparna_BS}.\\

%If $M$ is a trivial source module then it is isomorphic to a direct summand of a permutation module. Following similar arguments as in the proof of Theorem \ref{thmdim}, we can see that this invariant would be equal to the dimension of a certain submodule and hence will always be an integer.

This brings us to our next question which could possible classify all $\mathbf{k}G$-modules for a finite group $G$ that have integer invariant.

\begin{quest} What can we say about a $\mathbf{k}G$-module $M$ whose $\gamma_G(M)$ is an integer?
\end{quest}

%It is also interesting to analyse the dimension sequence of the core of higher tensor powers of a $kG$-module when $G$ is a finite group.\\
\smallskip
\smallskip

 \textbf{Acknowledgements:} I sincerely acknowledge the constant help and guidance of my advisor Prof. David J. Hemmer. I would like to thank him for a careful reading of this article and his valuable suggestions.\\
 \\
 
\bibliographystyle{plain}
\bibliography{bibliography}

\end{document}